\documentclass[11pt]{article}

\usepackage{amsmath,amssymb,amsthm}
\usepackage{bm}
\usepackage{microtype, array}

\newcommand{\absa}[1]{\left|#1\right|}
\newcommand{\absb}[1]{\lVert#1\rVert}
\newcommand{\norm}[1]{\left\lVert#1\right\rVert}
\newcommand{\rmd}{\mathrm{d}}  

\newtheorem{theorem}{Theorem}[section]
\newtheorem{corollary}[theorem]{Corollary}
\newtheorem{lemma}[theorem]{Lemma}
\newtheorem{proposition}[theorem]{Proposition}

\theoremstyle{definition}
\newtheorem{definition}[theorem]{Definition}

\title{Decay properties of Riesz transforms and steerable wavelets  \thanks{This research was funded in part by ERC Grant ERC-2010-AdG 267439-FUN-SP and the Swiss National Science Foundation under grant PBELP2-135867.}}

\author{
John Paul Ward \thanks{Biomedical Imaging Group, \'Ecole polytechnique f\'ed\'erale de Lausanne (EPFL),
Station 17, CH-1015, Lausanne, Switzerland  ({\tt john.ward@epfl.ch}).}
\and  
Kunal Narayan Chaudhury \thanks{Program in Applied and Computational Mathematics, Princeton University, USA ({\tt kchaudhu@math.princeton.edu})}   
\and
Michael Unser \thanks{Biomedical Imaging Group, \'Ecole polytechnique f\'ed\'erale de Lausanne (EPFL),
Station 17, CH-1015, Lausanne, Switzerland  ({\tt michael.unser@epfl.ch}).}
 }

\date{}

\begin{document}

\maketitle

\begin{abstract}
The Riesz transform is a natural multi-dimensional extension of the Hilbert transform, and it has been the object of study for many years due to its nice mathematical properties.  More recently, the Riesz transform and its variants have been used to construct complex wavelets and steerable wavelet frames in higher dimensions. The flip side of this approach, however, 
is that the Riesz transform of a wavelet often has slow decay. One can nevertheless overcome this problem by requiring the original wavelet to have sufficient smoothness, decay, and vanishing moments. In this paper, we derive necessary conditions in terms of these three properties that guarantee the decay of the Riesz transform and its variants, and as an application, we show how the decay of the popular Simoncelli wavelets can be improved by appropriately modifying their Fourier transforms. By applying the Riesz transform to these new wavelets, we obtain steerable frames with rapid decay.  
\end{abstract}



\section{Introduction}
The Riesz transform of a real-valued function $f(\bm{x})$ on $\mathbf{R}^d$ is the vector-valued function $\mathcal{R} f(\bm{x})= \left( \mathcal{R}_1f(\bm{x}), \ldots, \mathcal{R}_df(\bm{x}) \right)$ given by
\begin{equation}
\label{defn}
\mathcal{R}_if(\bm{x})=C_d \lim_{\epsilon\rightarrow 0} \int_{\absb{\bm{x}-\bm{y}}>\epsilon}f(\bm{y})\frac{x_i-y_i}{\absb{\bm{x}-\bm{y}}^{d+1}}\rmd\bm{y},
\end{equation}
where $C_d$ is a absolute constant depending on the dimension $d$ \cite{Duoandikoetxea}.  In this formula, we can see that the Riesz transform is defined by a principal value integral, as is the case for the Hilbert transform. Roughly speaking, the $i$-th component of the Riesz transform 
is obtained through the convolution of $f(\bm{x})$ with the kernel $x_i/\absb{\bm{x}}^{d+1}$, where $\absb{\bm{x}}$ is the Euclidean norm of $\bm{x}$, and $x_i$ are its coordinates.  The truncation of the integral in \eqref{defn} is used to contain the singularity of the kernel at the origin, 
by systematically ``chopping'' it off around the origin. This added detail makes the analysis of the transform quite involved.

Though the transform appears arcane at first sight, it tends to emerge naturally in various physical and mathematical settings due to its fundamental nature. It is essentially the unique
scalar-to-vector valued transform that is unitary (conserves energy), and is simultaneously invariant to the fundamental transformations of translation, 
scaling, and rotation. For example, the role of the Riesz transform in the theory of differential equations and 
multidimensional Fourier analysis has long been recognized in mathematics, where the above-mentioned invariances play a central role, see, e.g., \cite{Duoandikoetxea,Stein,Stein2}.  
The invariance of the transform to translations, namely that $\mathcal{R} [f(\cdot-\boldsymbol{u})](\bm{x})=[\mathcal{R} f](\bm{x}-\boldsymbol{u})$, is immediately seen from \eqref{defn}. 
On the other hand, its invariance to scaling is confirmed by the change of variable $\bm{x} \mapsto s \bm{x}$ in \eqref{defn}, where $s>0$.
The other properties are, however, not obvious at first sight. Instead, these are better seen from its Fourier transform, which is given by (see, e.g., \cite{Stein})
\begin{equation}
\label{fourier}
\widehat{\mathcal{R}_if}(\bm{\omega}) = -j \frac{\omega_i}{\absb{\bm{\omega}}}\widehat{f}(\bm{\omega}).
\end{equation}
From this and Parseval's theorem, we see the sum of the squared $\mathbf{L}^2$-norms of $\mathcal{R}_if(\bm{x})$ equals the squared $\mathbf{L}^2$-norm of $f(\bm{x})$. 
That is, the energy of $f(\bm{x})$ is divided between the different Riesz components. The Fourier expression also tells us that $f(\bm{x})$ can be recovered from its Riesz transform using a simple reconstruction, 
namely as
\begin{equation*}
f(\bm{x}) = -\sum_{i=1}^d \mathcal{R}_i [\mathcal{R}_if(\bm{x})].
\end{equation*}
Using the invariance of the Fourier transform with rotations, it can also be verified from \eqref{fourier} that the Riesz transform of the rotation of $f(\bm{x})$ is the (vector-field) rotation of $\mathcal{R} f(\bm{x})$.

The Riesz transform is a natural multidimensional extension of the Hilbert transform, obtained by setting $d=1$ in \eqref{defn}. The fact that 
the Hilbert transform plays a central role in determining the (complex) analytic counterpart of real-valued signals, and, particularly, in determining its instantaneous phase and frequency, has long been known 
in the signal processing and wavelet community, e.g., see \cite{Abry, Daubechies,Hahn,Olson}. The introduction of the Riesz transform to the community, 
however, has been more recent. In their pioneering work \cite{FelsbergSommer1,FelsbergSommer2}, Felsberg and Sommer used the transform to define a two dimensional generalization of the
analytic signal, which they called the \textit{monogenic signal}. At around this time,  Larkin et al. introduced a complexified version of the Riesz transform 
in two dimensions \cite{FelsbergLarkin1,FelsbergLarkin2}. They showed that this complex transform could be used for demodulating interferrograms, and for analyzing fringe patters in optics. Felsberg's work on the monogenic signal triggered a series of research works relating to the 
applications of the Riesz transform in the design of wavelets. For example, this inspired Metikas and Olhede to realize a monogenic version of the continuous wavelet 
transform in \cite{Metikas}. Further work, concerning filters based on Riesz transforms, can be found in \cite{Felsberg,Knutsson,Kothe}.

More recently, in a series of papers by Unser et al. \cite{Unser3,Unser1,Unser2}, it was shown how certain functional properties of 
the transform, particularly its invariances to translations, scaling and rotations, and its connection with derivatives and the Laplacian, could be used for realizing complex (monogenic) wavelet bases 
and \textit{steerable} wavelet frames in higher dimensions. This provided a natural functional counterpart to Simoncelli's steerable filterbanks \cite{Simoncelli}, which is a popular framework for decomposing images at
multiple scales and orientations. The key connection provided by Unser et al. was that the steerable wavelets of Simoncelli could be realized as the Riesz transform of appropriate bandpass functions. 
This powerful functional interpretation allowed them to enrich the wavelet design along several directions. For example, the authors proposed the so-called higher-order Riesz transforms obtained through
repeated applications of the various components of the Riesz transforms. They showed these higher-order transforms could be applied to a given wavelet to obtain a larger palette of steerable wavelets. When applied on the data, 
these frames act as (multi-scale) partial derivatives. The important point was that all of this could be done using fast and stable filterbank algorithms.

An aspect of this theory that has not previously been studied in depth is how the Riesz transform affects the decay of a wavelet.   
The importance of this issue lies in the fact that a slowly decaying wavelet can produce poor approximations that suffer from problems such as ringing artifacts and boundary effects.  Here, we propose that good decay of the Riesz transform $\mathcal{R} f(\bm{x})$ is determined by the smoothness, decay, and vanishing moments of $f(\bm{x})$.  In fact, this observation explains the poor spatial decay of some of the earlier wavelet constructions involving the Riesz transform. For example, Felsberg applied the transform on the Poisson kernel which has no vanishing moments \cite{FelsbergSommer1,FelsbergSommer2}, and Simoncelli used it on wavelets with poor spatial decay (due to the non-smooth cut-off in the spectrum) \cite{Simoncelli}.

In this paper, we perform a quantitative study of the decay property of the Riesz transform and its higher-order variants. We expand upon the work in \cite{Chaudhury}, 
where the intimate connection between the Hilbert transform (the one-dimensional Riesz transform) and wavelets was investigated. It was observed there that for a wavelet with large number of vanishing moments
 and a reasonably good decay, its Hilbert transform automatically has an ultra-fast decay. We extend the results in \cite{Chaudhury} first to the Riesz transform, and then to its higher-order counterparts. 
While the main principles are the same, we are required to use a different approach in the higher dimensional setting. The reason for this is that when $d=1$, the Riesz multiplier has the particularly simple form $1/\pi x$, which
 does not involve the term $\absa{x}$. Also, the concept of higher-order transforms is trivial for $d=1$, since by applying the Hilbert transform twice, we get back the negative of the original wavelet. 
The higher-order transforms are thus of significance only for $d \geq 2$. The analysis of these transforms presents additional technical challenges due to the asymmetry between the decay of the original wavelet and
its Riesz transform. As a consequence, the results for the higher-order transforms cannot simply be obtained by a direct application of the results for a single Riesz transform.

\section{Notation}

We will follow the notations in \cite{Unser3,Unser1,Unser2}. The components of the higher-order Riesz transforms are obtained by iterated applications of the component
transforms.  Thus, given a multi-index of non-negative integers  $\bm{\alpha}=(\alpha_1,\dots,\alpha_d)$, we define the higher-order Riesz transform by

\begin{equation*}
 \mathcal{R}^{\bm{\alpha}}f(\bm{x}) = \sqrt{\frac{\absa{\bm{\alpha}}!}{\bm{\alpha}!}}\mathcal{R}_1^{\alpha_1}\cdots \mathcal{R}_d^{\alpha_d}f(\bm{x}),
\end{equation*}
where $\mathcal{R}_i^k$ denotes the $k$-fold application of $\mathcal{R}_i$ to $f(\bm{x})$, and $\absa{\bm{\alpha}}$ is the order of the transform.  Here we are using the the notation $\absa{\bm{\alpha}}=\alpha_1+\cdots+\alpha_d$ and $\bm{\alpha}!=\alpha_1 !\cdots\alpha_d!$.
Also, we will write the partial derivative $D^{\bm{\alpha}}f(\bm{x})$ as

\begin{equation*}
D^{\bm{\alpha}} f (\bm{x})=\frac{\partial^{\absa{\bm{\alpha}}}}{\partial_{x_1}^{\alpha_1}\cdots\partial_{x_d}^{\alpha_d}} f (\bm{x}). 
\end{equation*}
for a given multi-index $\bm{\alpha}=(\alpha_1,\dots,\alpha_d)$. The inner-product between two vectors $\bm{x}$ and $\bm{y}$ will be denoted by $\bm{x}\cdot \bm{y}$, and $\bm{x}^{\bm{\alpha}}$ will represent the monomial $x_1^{\alpha_1} \cdots x_d^{\alpha_d}$.

 It is clear that the various components of $\mathcal{R} f(\bm{x})$ have similar analytic properties. We therefore fix $i$ once and for all and 
define the Riesz kernel $R(\bm{y})=y_i/\absb{\bm{y}}^{d+1}$. Our goal in this paper is to study the Riesz transforms of well-behaved functions,
with a certain number of vanishing moments. By well-behaved, we mean that they should be sufficiently smooth and must have enough decay. To this end, we define the class $L_N(\mathbf{R}^d)$ to be the collection of differentiable functions $f(\bm{x})$ such that

\begin{enumerate}
\item  $\absa{f(\bm{x})} \leq C(1+\absb{\bm{x}})^{-d-N+\epsilon}$,
\item  $\absa{ D^{\bm{\alpha}} f(\bm{x})}\leq C (1+\absb{\bm{x}})^{-d-N-1+\epsilon}, \ \absa{\bm{\alpha}}=1$, \text{ and}
\item  $\int \bm{x}^{\bm{\beta}} f(\bm{x})\rmd\bm{x}=0, \ \absa{\bm{\beta}} < N$.
\end{enumerate}
Here $C>0$ and $0 \leq \epsilon <1$ are arbitrary constants, and the class index $N \geq 1$ is an integer.  The first two properties address decay and smoothness, while the third condition guarantees vanishing moments.

\section{Riesz transform of wavelets in $L_N(\mathbf{R}^d)$} 

This section comprises the heart of the paper, as we establish results showing how the component Riesz transforms preserve decay and vanishing moments for functions of the above class. In particular, we will show that the transform $\mathcal{R}_if$ of a function  $f\in L_N(\mathbf{R}^d)$ has a rate of decay comparable to $f$ and has the same number of vanishing moments.  Significantly, these theorems are formulated in a way that allows then to be applied iteratively so that we may derive results for the higher-order transforms in the next section. However, as the transform is given in terms of a singular integral, before estimating its decay, we need to show that it is well-defined.

\begin{proposition}\label{pr:well_def}
For all $f \in L_N(\mathbf{R}^d)$, $\mathcal{R}_if$ is well-defined, continuous, and bounded by 
\begin{equation*}
C\left(\sup_{\bm{t}\in\mathbf{R}^d}\absb{\nabla f(\bm{t})}+\norm{f}_1\right).
\end{equation*}
for some $C>0$.
\end{proposition}
\begin{proof}
We can see that the domain of the singular integral defining $\mathcal{R}_if$ is unbounded, and we have an additional difficulty in a neighborhood of the point $\bm{x}$.  We shall apply different techniques to handle these problems, so we split the integral into two terms.  For a fixed $\bm{x}$, we have 
 
\begin{align*}
\lim_{\epsilon\rightarrow 0} \int_{\absb{\bm{x}-\bm{y}}>\epsilon}f(\bm{y})\frac{x_i-y_i}{\absb{\bm{x}-\bm{y}}^{d+1}}\rmd\bm{y} &= \lim_{\epsilon\rightarrow 0} \int_{1>\absb{\bm{x}-\bm{y}}>\epsilon}f(\bm{y})\frac{x_i-y_i}{\absb{\bm{x}-\bm{y}}^{d+1}}\rmd\bm{y}\\
&\quad + \int_{\absb{\bm{x}-\bm{y}}>1}f(\bm{y})\frac{x_i-y_i}{\absb{\bm{x}-\bm{y}}^{d+1}}\rmd\bm{y}.
\end{align*}
Using the fact that the Riesz kernel is odd and the inequality 

\begin{equation}\label{eq:mvt}
\frac{\absa{f(\bm{y})-f(\bm{x})}}{\absb{\bm{x}-\bm{y}}}\frac{\absa{x_i-y_i}}{\absb{\bm{x}-\bm{y}}}\frac{1}{\absb{\bm{x}-\bm{y}}^{d-1}} \leq \sup_{\bm{t}\in\mathbf{R}^d}\absb{\nabla f(\bm{t})}\frac{1}{\absb{\bm{x}-\bm{y}}^{d-1}},
\end{equation}
holds, we can write

\begin{equation}\label{eq:odd}
\int_{1>\absb{\bm{x}-\bm{y}}>\epsilon}f(\bm{y})\frac{x_i-y_i}{\absb{\bm{x}-\bm{y}}^{d+1}}\rmd\bm{y}=\int_{1>\absb{\bm{x}-\bm{y}}>\epsilon}(f(\bm{y})-f(\bm{x}))\frac{x_i-y_i}{\absb{\bm{x}-\bm{y}}^{d+1}}\rmd\bm{y}
\end{equation}
and apply the dominated convergence theorem, which shows that the first term is convergent. This also shows that the first term is bounded by $C\sup_{\bm{t}\in\mathbf{R}^d}\absb{\nabla f(\bm{t})}$.  For the second term, we have

\begin{equation*}
\frac{\absa{x_i-y_i}}{\absb{\bm{x}-\bm{y}}^{d+1}}\leq 1,
\end{equation*}
which implies it is bounded by $\norm{f}_1$.

To finish the proof, we must show that $\mathcal{R}_if$ is continuous, so we fix $\bm{x}\in\mathbf{R}^d$ and $\gamma>0$. 
Continuity can then be proved by showing that there are finite constants $C>0$ and $\delta>0$ such that $|\mathcal{R}_if(\bm{x})-\mathcal{R}_if(\bm{z})|<C\gamma$ 
for $\absb{\bm{x}-\bm{z}}<\delta$.  To that end, consider the transform of $f$  written as 

\begin{align*}
  \mathcal{R}_if(\bm{x}) &= \lim_{\epsilon\rightarrow 0} \int_{\epsilon'>\absb{\bm{x}-\bm{y}}>\epsilon}f(\bm{y})\frac{x_i-y_i}{\absb{\bm{x}-\bm{y}}^{d+1}}\rmd\bm{y}\\
  &\quad+
 \int_{\absb{\bm{x}-\bm{y}}>\epsilon'}f(\bm{y})\frac{x_i-y_i}{\absb{\bm{x}-\bm{y}}^{d+1}}\rmd\bm{y}
\end{align*}
for $\epsilon'>\epsilon>0$.  Using the same techniques as in equations \eqref{eq:mvt} and \eqref{eq:odd}, we can see that $\epsilon'$ should be chosen so that the absolute value of the first integral is smaller than $\gamma$.  Therefore 

\begin{equation*}
 \absa{\mathcal{R}_if(\bm{x})-\mathcal{R}_if(\bm{z})} \leq  2\gamma + \int_{\absb{\bm{x}-\bm{y}}>\epsilon'}  \absa{f(\bm{y})}\left|\frac{x_i-y_i}{\absb{\bm{x}-\bm{y}}^{d+1}}-\frac{z_i-y_i}{\absb{\bm{z}-\bm{y}}^{d+1}}\right|\rmd\bm{y},
\end{equation*}
and since the Riesz kernel is uniformly continuous\footnote{This is because it is differentiable away from the singularity, and its gradient is bounded.} away from its singularity, there is some $\delta>0$ (depending on $\epsilon'$) for which 

\begin{align*}
 \absa{\mathcal{R}_if(\bm{x})-\mathcal{R}_if(\bm{z})}&\leq \gamma\left(2 +  \int_{\absb{\bm{x}-\bm{y}}>\epsilon'}\absa{f(\bm{y})}\rmd\bm{y}\right)\\
 &\leq \gamma\left(2 +\norm{f}_1\right)\\
\end{align*}
when $\absb{\bm{x}-\bm{z}}<\delta$.
\end{proof}

Now that we have established the fact that the transform is well-defined,  our focus is shifted to the main result concerning decay. Our approach is similar to the one used by Adams in \cite{Adams} in order to bound weighted $L^p$ norms of Riesz transforms, and as in that paper, a key property is the number of vanishing moments.   As there are different issues with the singular integral depending on the region of $\mathbf{R}^d$,  we divide it into three pieces and bound each separately.  One of these regions will be handled by using the vanishing moment condition on $f$ and approximating the shifted version of the Riesz kernel $R(\bm{x}-\cdot)$ by its 
Taylor expansion at $\bm{0}$. Therefore, for a fixed $\bm{x}$, we denote the degree $N-1$ Taylor polynomial approximant to $R(\bm{x}-\cdot)$ by $P_N(\bm{x},\cdot)$ and define $\Phi_N(\bm{x},\bm{y}):=R(\bm{x}-\bm{y})-P_N(\bm{x},\bm{y})$.  With this notation, we have 

\begin{equation*}
P_N(\bm{x},\bm{y}) + \Phi_N(\bm{x},\bm{y}) = \sum_{\absa{\bm{\alpha}}\leq N-1} \frac{1}{\bm{\alpha} !}D^{\bm{\alpha}} R(\bm{x})\bm{y}^{\bm{\alpha}} + \sum_{\absa{\bm{\alpha}}= N} \frac{1}{\bm{\alpha} !} D^{\bm{\alpha}} R(\bm{x}-\bm{\xi})\bm{y}^{\bm{\alpha}},
\end{equation*}
where $\bm{\xi}$ is a point on the line segment from $\bm{0}$ to $\bm{y}$. Two estimates that will be used in the proof are

\begin{equation}\label{eq1}
\absa{D^{\bm{\beta}} R(\bm{x})}\leq C \absb{\bm{x}}^{-d-\absa{\bm{\beta}}}
\end{equation}
for $\bm{x}\neq \bm{0}$ and

\begin{equation}\label{eq2}
\absa{R(\bm{x}-\bm{y})-P_N(\bm{x},\bm{y})} \leq C \absb{\bm{x}}^{-(d+N)} \absb{\bm{y}}^{N}
\end{equation}
for $\absb{\bm{y}}\leq \absb{\bm{x}}/2$. The estimate {\eqref{eq1}} can be verified by differentiating the Riesz kernel $R(\bm{x})= x_i/\absb{\bm{x}}^{d+1}$, and {\eqref{eq2}} follows by applying {\eqref{eq1}} to the remainder in Taylor's theorem.

\begin{theorem}\label{th:decay}
 If $f\in L_N(\mathbf{R}^d)$, then the component transforms of $f$ 
 
\begin{equation*}
\mathcal{R}_if(\bm{x}) = C_d \lim_{\epsilon'\rightarrow 0} \int_{\absb{\bm{x}-\bm{y}}>\epsilon'}f(\bm{y})\Phi_N(\bm{x}-\bm{y})\rmd\bm{y} 
\end{equation*}
satisfy the bound

\begin{equation*}
 \absa{\mathcal{R}_if(\bm{x})} \leq C (1+\absb{\bm{x}})^{-d-N+\epsilon+\delta}
\end{equation*}
for some $0< \delta <1-\epsilon$.
\end{theorem}

\begin{proof}
The upper bound was proved in Proposition \ref{pr:well_def}, so it suffices to verify that $\mathcal{R}_if(\bm{x})$ has sufficient decay for $\absb{\bm{x}}$ large.  In order to prove the bound, we divide the domain of the integral defining $\mathcal{R}_if$ into the regions

\begin{align*} 
D_1 &= \{\absb{\bm{y}}<\absb{\bm{x}}/2\}, \\
D_2 &= \{\absb{\bm{y}-\bm{x}} \geq \absb{\bm{x}}/2, \absb{\bm{y}} \geq \absb{\bm{x}}/2\}, \\
D_3 &= \{\absb{\bm{y}-\bm{x}}<\absb{\bm{x}}/2\}.
\end{align*}
The regions $D_1$ and $D_3$ are disjoint open balls centered at the origin and $\bm{x}$ respectively, and $D_2$ is the complement of their union.  Since we assume $\absb{\bm{x}}>0$, these sets form a partition of $\mathbf{R}^d$.  First, we consider the region $D_1$.  Using equation \eqref{eq2} we get

\begin{equation*}
\int_{D_1}\absa{f(\bm{y})\Phi_N(\bm{x},\bm{y})}\rmd\bm{y}  \leq C\int_{D_1}\absa{f(\bm{y})}\absb{\bm{y}}^{N}\absb{\bm{x}}^{-d-N}\rmd\bm{y}, \\
\end{equation*}
and since $2<\absb{\bm{x}}/\absb{\bm{y}}$, the inequality

\begin{equation*}
\int_{D_1}\absa{f(\bm{y})\Phi_N(\bm{x},\bm{y})}\rmd\bm{y} \leq C\absb{\bm{x}}^{-d-N+\epsilon+\delta}\norm{\absb{\cdot}^{N-\epsilon-\delta}f(\cdot)}_1
\end{equation*}
holds for any positive $\delta$. Furthermore, due to the decay of $f$, the norm in the right hand side of the above equation is finite.

Next, consider the region $D_2$, where we may apply equation \eqref{eq1} to obtain 

\begin{align*}
\int_{D_2}\absa{f(\bm{y})\Phi_N(\bm{x},\bm{y})}\rmd\bm{y} &\leq C\int_{D_2}\absa{f(\bm{y})}\left(\absb{\bm{x}}^{-d}+\sum_{k=0}^{N-1}\absb{\bm{y}}^k\absb{\bm{x}}^{-d-k}\right)\rmd\bm{y} \\
&\leq C\absb{\bm{x}}^{-d}\int_{D_2}\absa{f(\bm{y})}\left(\sum_{k=0}^{N-1}\frac{\absb{\bm{y}}^k}{\absb{\bm{x}}^k}\right)\rmd\bm{y},
\end{align*}
and note that for $k\leq N-1$, we have $N-(\epsilon+\delta)-k>0$, since $0\leq \epsilon+\delta <1$. Additionally, $2\absb{\bm{y}} \geq \absb{\bm{x}}$ in this region, so we multiply each term of the sum by $(2\absb{\bm{y}}/\absb{\bm{x}})^{N-(\epsilon+\delta)-k}$ to obtain the required bound:

\begin{align*}
\int_{D_2}\absa{f(\bm{y})\Phi_N(\bm{x},\bm{y})}\rmd\bm{y} &\leq  C \absb{\bm{x}}^{-d}\int_{D_2}\absa{f(\bm{y})} \left(\frac{\absb{\bm{y}}}{\absb{\bm{x}}}\right)^{N-\epsilon-\delta}\rmd\bm{y}\\
&\leq C\absb{\bm{x}}^{-d-N+\epsilon+\delta} \norm{\absb{\cdot}^{N-\epsilon-\delta}f(\cdot)}_1.
\end{align*}

Finally, the region $D_3$ contains the singularity of the Riesz kernel, and we further divide the integral:

\begin{align*}
\lim_{\epsilon' \rightarrow 0} \int_{D_3\backslash B(\bm{x},\epsilon') }f(\bm{y})\Phi_N(\bm{x},\bm{y})\rmd\bm{y} =& \lim_{\epsilon' \rightarrow 0} \int_{D_3\backslash B(\bm{x},\epsilon') }f(\bm{y})R(\bm{x}-\bm{y})\rmd\bm{y} \\
&- \int_{D_3}f(\bm{y})P_N(\bm{x},\bm{y})\rmd\bm{y},
\end{align*}
where $B(\bm{x},\epsilon')$ denotes the ball of radius $\epsilon'$ centered at $\bm{x}$.  The second part is bounded using techniques similar to those used in the region $D_2$:

\begin{align*}
\int_{D_3}\absa{f(\bm{y})P_N(\bm{x},\bm{y})}\rmd\bm{y} &\leq C\int_{D_3}\absa{f(\bm{y})}\sum_{k=0}^{N-1}\absb{\bm{y}}^{k}\absb{\bm{x}}^{-d-k}\rmd\bm{y} \\
&\leq C \int_{D_3}\absa{f(\bm{y})} \absb{\bm{x}}^{-d} \left(\frac{\absb{\bm{y}}}{\absb{\bm{x}}}\right)^{N-1}\rmd\bm{y} \\
&\leq C \absb{\bm{x}}^{-N-d+\epsilon+\delta} \norm{\absb{\cdot}^{N-\epsilon-\delta}f(\cdot)}_1.
\end{align*}
The first part is bounded using the mean value theorem along with the estimate for the decay of the derivatives of $f$:

\begin{align*}
\absa{\int_{D_3\backslash B(\bm{x},\epsilon')}f(\bm{y})R(\bm{x}-\bm{y})\rmd\bm{y}} &=  \absa{\int_{D_3\backslash B(\bm{x},\epsilon')}\frac{(f(\bm{x})-f(\bm{y}))}{\absb{\bm{x}-\bm{y}}}\frac{(x_i-y_i)}{\absb{\bm{x}-\bm{y}}^{d}}\rmd\bm{y}}\\
&\leq C \sup_{\bm{t} \in B(\bm{x},\absb{\bm{x}}/2)}\absb{\nabla f(\bm{t})} \int_{D_3}\absb{\bm{x}-\bm{y}}^{-d+1}\rmd\bm{y} \\
&\leq C (1+\absb{\bm{x}})^{-d-N+\epsilon}.
\end{align*}
As this bound is independent of $\epsilon'$, the result follows.
\end{proof}

As previously stated, the Riesz transform preserves both decay and vanishing moments, and both are vital for deriving decay estimates for the higher-order Riesz transforms.  The former property was the subject of the last theorem, and we now address the latter; i.e., we will show that for any $f\in L_N(\mathbf{R}^d)$, $\mathcal{R}_if$ has vanishing moments up to degree $N-1$.  Basically, the idea is to consider both the Fourier and spatial domain formulations concurrently.  To wit, vanishing moments can be interpreted both as the order of the zero of a function at the origin in the Fourier domain and as an integral.

\begin{proposition}\label{pr:van}
For $f \in L_N(\mathbf{R}^d)$ and $\absa{\bm{\beta}}<N$, we have 
\begin{equation*}
D^{\bm{\beta}} \widehat{f}(\bm{\omega}) = o(\absb{\bm{\omega}}^{N-1-\absa{\bm{\beta}}}) \text{\hspace*{1cm} as \hspace*{1cm}} \bm{\omega}\rightarrow \bm{0}.
\end{equation*}
\end{proposition}

\begin{proof}
The decay estimate on $f$ implies that $(\cdot)^{\bm{\beta}}f$ is in $L^1(\mathbf{R}^d)$ for $\absa{\beta} < N$,  so $\widehat{f}$ has continuous derivatives up to order $N-1$ and
\begin{equation*}
D^{\bm{\beta}} \widehat{f} = C[(\cdot)^{\bm{\beta}}f]^\wedge,
\end{equation*}
cf. \cite[Theorem 1.1.7]{Stein}. The right-hand side can be expressed as the integral 

\begin{equation*}
D^{\bm{\beta}} \widehat{f}(\bm{\omega}) = C\int \bm{x}^{\bm{\beta}} f(\bm{x}) e^{j \bm{x}\cdot \bm{\omega}}\rmd\bm{x},
\end{equation*}
and the vanishing moment condition on $f$ implies that $D^{\bm{\beta}}\widehat{f}(\bm{0})=0$ for $\absa{\bm{\beta}}<N$.  Therefore, Taylor's theorem implies that

\begin{equation*}
D^{\bm{\beta}} \widehat{f}(\bm{\omega}) = \sum_{\absa{\bm{\alpha}}= N-1-\absa{\bm{\beta}}} \frac{1}{\bm{\alpha} !} D^{\bm{\alpha}+\bm{\beta}} \widehat{f}(\bm{\xi}) \bm{\omega}^{\bm{\alpha}}
\end{equation*}
for some $\bm{\xi}$ on the line segment from $\bm{0}$ to $\bm{\omega}$.  The result now follows from the fact that $D^{\bm{\alpha}+\bm{\beta}}\widehat{f}(\bm{\xi})\rightarrow 0$ as $\bm{\xi}\rightarrow \bm{0}$.
\end{proof}

\begin{theorem}\label{th:mom}
Given any $f \in L_N(\mathbf{R}^d)$, $\mathcal{R}_if$ has vanishing moments up to degree $N-1$. 
\end{theorem}

\begin{proof}
The proof relies on considering two representations of $\widehat{\mathcal{R}_if}$.  First we have

\begin{equation*}
\widehat{\mathcal{R}_if}(\bm{\omega}) = C \int \mathcal{R}_if(\bm{x})e^{j\bm{x}\cdot\bm{\omega}}\rmd\bm{x},
\end{equation*}
so

\begin{equation*}
D^{\bm{\beta}}\widehat{\mathcal{R}_if}(\bm{\omega}) = C \int \bm{x}^{\bm{\beta}} \mathcal{R}_if(\bm{x})e^{j\bm{x}\cdot\bm{\omega}}\rmd\bm{x}.
\end{equation*}
We also know that

\begin{equation*}
\widehat{\mathcal{R}_if}(\bm{\omega})=\frac{\omega_i}{\absb{\bm{\omega}}}\widehat{f}(\bm{\omega}),
\end{equation*}
and hence

\begin{equation*}
D^{\bm{\beta}}\widehat{\mathcal{R}_if}(\bm{0})=0
\end{equation*}
for $\absa{\bm{\beta}}<N$ by Proposition \ref{pr:van}.  Putting the two together, we have

\begin{equation*}
0 = D^{\bm{\beta}}\widehat{\mathcal{R}_if}(\bm{0}) = C \int \bm{x}^{\bm{\beta}} \mathcal{R}_if(\bm{x})\rmd\bm{x}
\end{equation*}
\end{proof}

Let us conclude this section by considering some examples.  To see how our results can be applied, we will first examine the Hermitian wavelets in one dimension, which we recall are derivatives of the Gaussian,
\begin{equation*}
\psi_n(t) = C_n \frac{d^n}{dt^n} e^{-t^2/2},
\end{equation*}
where $C_n$ is a normalization constant.  Clearly, these wavelets are infinitely smooth and decay faster than any polynomial.  Therefore, we must look at their vanishing moments to determine the spaces $L_N$ in which they lie.  Being derivatives of the Gaussian, integration by parts shows that $\psi_n$ has $n$ vanishing moments, so $\psi_n\in L_n(\mathbf{R})$.  Applying the results of this section, we conclude that the Riesz transform of $\psi_n$, denoted by $\mathcal{R}\psi_n$, satisfies: $\absa{\mathcal{R}\psi_n (x)} \leq C(1+\absa{x})^{-n}$ and $\mathcal{R}\psi_n$ has $n$ vanishing moments.  Note that since we are only interested in the essential decay of the Riesz transform, we have bounded the term $\absa{x}^{\delta+\epsilon}$ by $\absa{x}$ in our estimate.  However, the more precise decay rate stated in Theorem \ref{th:decay} will be important for showing how decay is preserved for higher-order Riesz transforms. 

As a second example, consider the two dimensional Poisson kernel
\begin{equation*}
p(\bm{x};s) = \frac{s}{2\pi(s^2+\absb{\bm{x}}^2)^{3/2}},
\end{equation*}
which was studied by Felsberg and Larkin \cite{FelsbergLarkin1,FelsbergLarkin2}.  The components of the Riesz transform of this function are computed to be
\begin{equation*}
\mathcal{R}_ip(\bm{x};s) = C_d \frac{x_i}{(s^2+\absb{\bm{x}}^2)^{3/2}},
\end{equation*}
which have decay comparable to $\absb{\bm{x}}^{-2}$ in the direction $x_i$.  Here, we can again see that while the Poisson kernel is smooth, its Riesz transform has slower decay due to the fact that $p(\bm{x};s)$ has no vanishing moments.

\section{Higher-order Riesz transforms}

In the previous section, we showed that a component transform of a function in $L_N(\mathbf{R}^d)$ has a rate of decay similar to the original function and a
comparable number of vanishing moments.  With this in mind, we would like to extend these results to
higher-order Riesz transforms by applying the results for the single order transform in an iterated fashion. In order to accomplish this, we will need to impose conditions which are stronger than those
required to be in $L_N(\mathbf{R}^d)$.
To make this precise, we recall the definition of the Sobolev spaces on $\mathbf{R}^d$.  

\begin{definition}
Let $H^s(\mathbf{R}^d)$ denote the $L^2$ Sobolev space of order $s\geq 0$ on $\mathbf{R}^d$, i.e. 
\begin{equation*}
 H^s(\mathbf{R}^d)=\left\{f\in L^2(\mathbf{R}^d): \widehat{f}(\cdot)(1+\norm{\cdot}_2^2)^{s/2}\in L^2(\mathbf{R}^d)  \right\}.
\end{equation*}

\end{definition}

The additional properties that we required of the original wavelet are given in terms its derivatives.  In particular, the derivatives of the wavelet should also have decay and vanishing moments.
\begin{definition}
 For $N \in \mathbf{N}$ and $M \in \mathbf{N}$ we define the space $L_{N,M}(\mathbf{R}^d)$ to be the collection of $f:\mathbf{R}^d\rightarrow \mathbf{R}$ satisfying:
\begin{enumerate}
 \item $f\in H^s(\mathbf{R}^d)\cap C(\mathbf{R}^d)$ for some $s>M+d/2$
 \item There are constants $C>0$ and $0 \leq \epsilon <1$ such that for any $0 \leq \absa{\bm{\alpha}} \leq M$ we have
\begin{enumerate}
\item  $\absa{ D^{\bm{\alpha}} f(\bm{x})}\leq C (1+\absb{\bm{x}})^{-d-N-\absa{\bm{\alpha}}+\epsilon}$  for $0 \leq \absa{\bm{\alpha}} \leq M$
\item  $\int \bm{x}^{\bm{\beta}} D^{\bm{\alpha}} f(\bm{x})\rmd\bm{x}=0$ for $0 \leq \absa{\bm{\alpha}} \leq M-1$ and $\absa{\bm{\beta}} < N+\absa{\bm{\alpha}}$
\end{enumerate}
\end{enumerate}
\end{definition}

Notice that the first property in the definition implies $L_{N,M}(\mathbf{R}^d)$ is contained in $C^M(\mathbf{R}^d)$ by the Sobolev embedding theorem.  Our definition also guarantees that $L_{N,M}(\mathbf{R}^d)$ is a subset of $L_N(\mathbf{R}^d)$ for $M\geq 1$, and we will show that the component transforms map $L_{N,M}(\mathbf{R}^d)$ to $L_{N,M-1}(\mathbf{R}^d)$. Hence the iterated
 transforms of a function in $L_{N,M}(\mathbf{R}^d)$ will have a rate of decay comparable to that of the the single transform
 of a function in $L_N(\mathbf{R}^d)$.  The proof relies on the iterated application of Theorem \ref{th:decay}
 and Theorem \ref{th:mom};  however, we must first establish that the transforms commute with differentiation.  

\begin{lemma}\label{lm:commute}
  If $f\in L_{N,M}(\mathbf{R}^d)$ for some $M>1$, then $\mathcal{R}_iD^{\bm{\alpha}} f=D^{\bm{\alpha}} \mathcal{R}_if$ for any $0<\absa{\bm{\alpha}}\leq M-1$.
\end{lemma}

\begin{proof}
 First, since $f\in H^s$ for some $s>M+d/2$, the Sobolev embedding theorem implies $f\in C^{M}(\mathbf{R}^d)$.
 Additionally, 
 
\begin{equation*}
 \widehat{\mathcal{R}_if} = \frac{\omega_i}{\absb{\bm{\omega}}}\widehat{f},
\end{equation*}
so $\mathcal{R}_if$ is also in $H^s$ and consequently in $C^{M}(\mathbf{R}^d)$.  Therefore for any fixed
 $0<\absa{\bm{\alpha}} \leq M-1$, $ \mathcal{R}_iD^{\bm{\alpha}} f$ and $D^{\bm{\alpha}} \mathcal{R}_if$ are continuous functions.  Furthermore, they
 are equivalent as distributions since they have the same Fourier transform 
$\omega_i/\absb{\bm{\omega}}\bm{\omega}^{\bm{\alpha}}\widehat{f}$. Therefore they are equivalent as functions, \cite[Theorem, p.27]{Kanwal}.
\end{proof}

Combining this lemma with the results of the previous section, we can see how the defining properties of $L_{N,M}$ ensure decay of the higher-order Riesz transforms, as we explicitly show in the next few results.
\begin{lemma}
 If $f\in L_{N,M}(\mathbf{R}^d)$ for some $M>1$, then $\mathcal{R}_if\in L_{N,M-1}(\mathbf{R}^d)$.
\end{lemma}

\begin{proof}
By Lemma \ref{lm:commute}, it suffices to verify the second property in the definition of $L_{N,M-1}(\mathbf{R}^d)$.  To begin, fix a multi-index $\bm{\alpha}$ with $0\leq \absa{\bm{\alpha}} \leq M-1$.  Since $f\in L_{N,M}(\mathbf{R}^d)$, we know 

\begin{enumerate}
 \item $\absa{ D^{\bm{\alpha}} f(\bm{x})}\leq C (1+\absb{\bm{x}})^{-d-N-\absa{\bm{\alpha}}+\epsilon}$ 
 \item $\absa{ D^{\bm{\alpha}+\bm{\gamma}} f(\bm{x})}\leq C (1+\absb{\bm{x}})^{-d-N-\absa{\bm{\alpha}}-1+\epsilon}$ for $\absa{\bm{\gamma}}=1$
 \item $D^{\bm{\alpha}} f$ has vanishing moments up to degree $N-1+\absa{\bm{\alpha}}$.
\end{enumerate}
Theorem \ref{th:decay} then implies that 

\begin{equation*}
  D^{\bm{\alpha}} \mathcal{R}_if(\bm{x})=\mathcal{R}_i D^{\bm{\alpha}} f(\bm{x})\leq C (1+\absb{\bm{x}})^{-d-N-\absa{\bm{\alpha}}+\epsilon+\delta}
\end{equation*}
for some $0<\delta<1-\epsilon$.  Additionally, Theorem \ref{th:mom} implies that $D^{\bm{\alpha}} \mathcal{R}_if(\bm{x})$ has vanishing moments up to degree $N-1+\absa{\bm{\alpha}}$.
\end{proof}

\begin{theorem}\label{th:dec}
 If $f\in L_{N,M}(\mathbf{R}^d)$, then for some $0<\epsilon<1$ and any $1\leq \absa{\bm{\alpha}} \leq M$ we have 
$\mathcal{R}^{\bm{\alpha}} f(\bm{x})\leq C (1+\absb{\bm{x}})^{-d-N+\epsilon}$ and $\mathcal{R}^{\bm{\alpha}} f(\bm{x})$ has vanishing moments up to degree $N-1$.
\end{theorem}

To illustrate this result, we consider a two-dimensional analog of the Hermitian wavelets: 
\begin{equation*}
\psi_n(\bm{x}) = C_{n} D^{\bm{\beta}} e^{-\absb{\bm{x}}^2/2},
\end{equation*}
where $\beta_1=\beta_2=2n$ for some positive integer $n$.  As before, the number of vanishing moments is the limiting factor when classifying $\psi_n$ in terms of the spaces $L_{N,M}$.  Considering their Fourier transforms, we can see that $\psi_n$ is in $L_{n,n}$.  Therefore, Theorem \ref{th:dec} implies that the $n$th-order Riesz transform of $\psi_n$, denoted by $\mathcal{R}^{n}\psi_n$, decays at least as fast as $(1+\absb{\bm{x}})^{-n}$.  Moreover, notice that if $\psi_n$ had infinitely many vanishing moments, this theorem would guarantee a rate of decay faster than any polynomial.  In fact, the Riesz transform preserves the collection of Schwartz class functions with infinitely many vanishing moments, as is stated formally in the following corollary.

\begin{corollary}\label{cor:sch}
 If $\psi$ is a Schwartz class function and all of its moments vanish, then $\mathcal{R}_i\psi$ satisfies
 the same properties.
\end{corollary}

\section{Application to Simoncelli wavelets}

As we have seen,  spatial decay of the Riesz transform is guaranteed by assuming smoothness, vanishing moments, and decay of the original function.    There are several examples of wavelets that are bandlimited and whose Fourier transform is zero in a neighborhood of the origin (cf. \cite{Unser3}).  The one property that many wavelet constructions lack is fast spatial decay.  This is a result of constructing wavelets in the Fourier domain using a non-smooth cut-off function.  One way to see this is to consider a Fourier integral representation of a wavelet $\psi$:
\begin{equation*}
\psi(x)=\frac{1}{(2\pi)^{1/2}} \int_{\mathbf{R}} \widehat{\psi}(\omega)e^{j \omega x}\rmd\omega.
\end{equation*}
If $\widehat{\psi}$ has $n$ continuous derivatives and compact support, then we can integrate by parts to get the decay estimate
\begin{equation*}
\absa{\psi(x)}=\frac{1}{(2\pi)^{1/2}} \frac{\norm{\widehat{\psi}^{(n)}}_1}{\absa{x}^n}
\end{equation*}
However, a discontinuity of a lower order derivative of $\widehat{\psi}$ can produce terms with slower decay.  This estimate also shows how the size of the derivatives of $\widehat{\psi}$ play a role in the decay of $\psi$.  To give a specific example and show how decay can be improved, we consider 
the Simoncelli wavelets \cite{Simoncelli1}, defined radially in the Fourier domain by 

\begin{equation*}
\widehat{\psi}(\omega) =
     \left\{
     \begin{array}{lr}
       \cos(\frac{\pi}{2}\log_2(\frac{2\omega}{\pi})), &  \omega \in (\frac{\pi}{4}, \pi  ]\\
       0, & \text{otherwise}
     \end{array}
   \right.
\end{equation*}
Note that this wavelet has vanishing moments of all degrees and finite Sobolev norm of all orders.  However, the lack of smoothness in the Fourier domain causes $\psi$ to have slow decay. With Theorem \ref{th:dec} as a guideline, we propose a modification to improve the decay of its higher Riesz transforms.

If we write $\widehat{\psi}$ as $(\chi_{[-1,1]}\cos(\frac{\pi}{2}\cdot))\circ \log_2(\frac{2}{\pi}\cdot)$, where $\circ$ denotes function composition, it is clear that we can create a smooth analog if we replace the indicator function $\chi$ by a smooth Meyer window.  In particular, let $\theta_\epsilon$ be a function satisfying:

\begin{enumerate}
\item $\theta_\epsilon$ is identically equal to $1$ on the interval $[-1+\epsilon,1-\epsilon]$
\item $\theta_\epsilon$ is  equal to $0$ outside $[-1-\epsilon,1+\epsilon]$
\item $\theta_\epsilon$ has at least $r$ continuous derivatives for some $r\in \mathbb N$
\item $\theta_\epsilon(x)^2+\theta_\epsilon(x-2)^2=1$ for $0\leq x \leq 2$.
\end{enumerate}
Now, if we merely replace the indicator function with $\theta_\epsilon$, we will have increased the support of of our wavelet in the Fourier domain and the basis it generates might no longer be a tight frame.  To compensate, we rescale so that support is maintained within $[-\pi,\pi]^d$. Let us denote the modified wavelet by $\psi'_\epsilon$, which in the Fourier domain is defined radially by 

\begin{equation*}
\left(\theta_\epsilon\cos\left(\frac{\pi}{2}\cdot\right)\right)\circ \log_2\left(\frac{2^{1+\epsilon}}{\pi}\cdot\right).
\end{equation*} 
The decay of this new function and its transforms can be determined by the results of the previous sections.

\begin{proposition}
The value of $r$ defining the smoothness of $\theta_\epsilon$ will determine the decay of $\psi_\epsilon'$, and the higher Riesz transforms of $\psi_\epsilon'$ will maintain a similar rate of decay.
\end{proposition}
\noindent
Here, we should point out that one of the nice properties of the Simoncelli wavelets is that they generate a tight frame, and importantly, our modification does not nullify this property.

\begin{proposition}
The wavelet basis constructed from the modified Simoncelli wavelet $\psi'_\epsilon$ is a tight frame.
\end{proposition}
\begin{proof}
See Appendix \ref{ap:proof52}.
\end{proof}

\setlength{\tabcolsep}{10pt}
\setlength{\extrarowheight}{5pt}

\begin{table}
\caption{\textbf{Examples of finite smoothness Meyer windows:}  Replacing $G$ by $G_n$ in equation \eqref{eq:H} and using the resulting $H_{n,\epsilon}$ in equation \eqref{eq:th}, we obtain a $ \theta_{n,\epsilon}$ in $C^{n}(\mathbf{R})$.}
\label{tab:mey1}
\begin{center}
\begin{tabular}{c|l}
$n$  & $G_n(x)$ \\ \hline \hline
    3  &  $\frac{35\pi}{64} (-\frac{1}{7}x^7+\frac{3}{5}x^5-x^3+x+\frac{16}{35})$ \\
    4  &  $\frac{315\pi}{512} (\frac{1}{9}x^9-\frac{4}{7}x^7+\frac{6}{5}x^5-\frac{4}{3}x^3+x+\frac{128}{315})   $ \\ 
    5  &  $\frac{693\pi}{1024} (-\frac{1}{11}x^{11}+\frac{5}{9}x^9-\frac{10}{7}x^7+2x^5-\frac{5}{3}x^3+x+\frac{256}{693})   $
\end{tabular}
\end{center}
\end{table}

We conclude now with a discussion about the construction of the Meyer windows $\theta_\epsilon$. The starting point is a smooth, non-negative bump function $g$ that is supported in $[-1,1]$ and whose definite integral is $\pi/2$.  From this function, we define

\begin{equation*}
G(\omega):= \int_{-\infty}^\omega g(t)\rmd t
\end{equation*}
and 

\begin{equation}\label{eq:H}
H_\epsilon(\omega):= G\left(\frac{\omega+1}{\epsilon}\right)-\frac{\pi}{2} + G\left(\frac{\omega-1}{\epsilon}\right)
\end{equation}
The Meyer window can then be constructed as

\begin{equation}\label{eq:th}
\theta_\epsilon(\omega) =\cos( H_\epsilon(\omega)).
\end{equation}
Indeed, one can verify that for $0 \leq \omega \leq 2$ we have
\begin{align*}
\theta_\epsilon(\omega)^2+\theta_\epsilon(\omega-2)^2 &= \cos^2(H_\epsilon(\omega)) +\cos^2(H_\epsilon(\omega-2)) \\
&= \cos^2(H_\epsilon(\omega)) +\cos^2(H_\epsilon(\omega)-\pi/2) \\
&= \cos^2(H_\epsilon(\omega)) +\sin^2(H_\epsilon(\omega)) \\
&= 1
\end{align*}
In the case of finite smoothness, examples can be found by setting $g_n(x)=C(1-x^2)^n$ on $[-1,1]$ and $0$ otherwise.  For some particular examples, see Table \ref{tab:mey1}.  In fact, the functions in the table are also derived in \cite{Held}; however, the authors' construction is different from ours. Notice that each such $G_n$ has $n$ continuous derivatives, so the corresponding $\theta_{n,\epsilon}$ and $\widehat{\psi}'_{n,\epsilon}$ will also have $n$ continuous derivatives, which translates into decay of the wavelet in the spatial domain.

In order to construct an infinitely smooth $\theta$, we take another approach, as explained in \cite{Held}.  Instead of starting with a bump function, we directly define the function $G$ used in \eqref{eq:H}.  Specifically, a function $G$ with the required properties is given by 
\begin{equation*}
G(\omega)=\frac{\pi}{2}\frac{\lambda(\omega+1)}{\lambda(\omega+1)+\lambda(1-\omega)}
\end{equation*}
where 
\begin{equation*}
\lambda(\omega)=
\begin{cases}
e^{-\omega^{-2}}, & \omega>0 \\
0, &\text{otherwise},
\end{cases}
\end{equation*}
and the result of this construction is a Meyer window with infinite smoothness.

\section{Conclusion}
In this paper, we have derived decay estimates for higher-order Riesz transforms of wavelets, where the conditions are given in terms of smoothness, decay and vanishing moments.  In the case of a single Riesz transform, it was shown that the decay rate of the original function is essentially maintained by the component transforms, and this result is likely to be optimal.  These results instill confidence in recently proposed steerable wavelet constructions, as the Riesz transform plays a large role in the design. Additionally, they point on the importance of specifying primal wavelets with a large number of vanishing moments, a fast rate of decay, and a sufficient amount of smoothness.  In particular, our results indicate that having many vanishing moments is not sufficient to guarantee decay; a matched degree of smoothness in both the space and frequency domains is also required.  Optimal conditions for the decay estimates of the higher-order transforms remains an open question.
By using the single-order results in an iterated fashion, we may be imposing conditions that are stronger than necessary, and taking a more direct approach may provide improved results.  Given the importance of the Riesz transform in steerable wavelet constructions, we view this as a topic that warrants further investigation.

\appendix

\section{Proof of Proposition 5.2}\label{ap:proof52}
Due to the support of  $\widehat{\psi}'_\epsilon$, it suffices to verify 

\begin{equation*}
\sum_{k\in \mathbb Z}{\absa{\widehat{\psi}'_\epsilon(2^k\omega)}^2}=1
\end{equation*}
for $\omega>0$. Using the definition, we have

\begin{equation*}
\sum_{k\in \mathbb Z}{\absa{\widehat{\psi}'_\epsilon(2^k\omega)}^2} =\sum_{k\in \mathbb Z}\absa{ \left(\theta_\epsilon\cos\left(\frac{\pi}{2}\cdot\right)\right)\circ \log_2\left(\frac{2^{1+\epsilon}}{\pi} 2^k \omega\right)}^2, 
\end{equation*}
and setting $z=\log_2(2^{1+\epsilon}\omega/\pi)$ gives

\begin{equation*}
\sum_{k\in \mathbb Z}{\absa{\widehat{\psi}'_\epsilon(2^k\omega)}^2} = \sum_{k\in \mathbb Z}\absa{ \theta_\epsilon(z+k)\cos\left(\frac{\pi}{2}(z+k)\right)}^2.
\end{equation*}
This sum is evaluated by considering two cases that depend on $z$.  First, for $-\epsilon \leq z \leq \epsilon$, we have

\begin{align*}
\sum_{k\in \mathbb Z}{\absa{\widehat{\psi}'_\epsilon(2^k\omega)}^2} &=\sum_{k=-1}^1 \absa{ \theta_\epsilon(z+k)\cos\left(\frac{\pi}{2}(z+k)\right)}^2\\
&= \cos^2\left(\frac{\pi}{2}z\right)+\cos^2\left(\frac{\pi}{2}(z-1)\right)\theta_\epsilon(z-1)^2 \\
&\quad +\cos^2\left(\frac{\pi}{2}(z+1)\right)\theta_\epsilon(z+1)^2 \\
&= \sin^2\left(\frac{\pi}{2}(z-1)\right) \\
&\quad+ \left( \theta_\epsilon(z-1)^2  + \theta_\epsilon(z+1)^2 \right)\cos^2\left(\frac{\pi}{2}(z-1)\right) \\
&= 1.
\end{align*}
Next, for $\epsilon \leq z \leq 1-\epsilon$, we have

\begin{align*}
\sum_{k\in \mathbb Z}{\absa{\widehat{\psi}'_\epsilon(2^k\omega)}^2} &=\sum_{k=-1}^0 \absa{ \theta_\epsilon(z+k)\cos\left(\frac{\pi}{2}(z+k)\right)}^2\\
&= \cos^2\left(\frac{\pi}{2}(z-1)\right)+\cos^2\left(\frac{\pi}{2}z\right) \\
&= \sin^2\left(\frac{\pi}{2}z\right) + \cos^2\left(\frac{\pi}{2}z\right) \\
&= 1.
\end{align*}

\bibliographystyle{plain}

\end{document}